\documentclass[11pt]{article}

\usepackage{amsmath, amsthm, amssymb}
\usepackage{fullpage}

\theoremstyle{plain}
\newtheorem{thm}{Theorem}[section]
  \theoremstyle{plain}
  
  \theoremstyle{plain}
  \newtheorem{lem}[thm]{Lemma}
  \newtheorem{conjecture}[thm]{Conjecture}

\date{}
\title{\vspace{-0.7cm}Maximum union-free subfamilies}
\author{
Jacob Fox \thanks{Department of Mathematics, MIT, Cambridge, MA 02139-4307. Email: fox@math.mit.edu. Research
supported by a Simons Fellowship.}
\and
Choongbum Lee \thanks{Department of Mathematics, UCLA, Los
Angeles, CA, 90095. Email: choongbum.lee@gmail.com. Research
supported in part by a Samsung Scholarship.}
\and
Benny Sudakov \thanks{Department of Mathematics, UCLA, Los Angeles, CA 90095.
Email: bsudakov@math.ucla.edu. Research supported
in part by NSF CAREER award DMS-0812005 and by a USA-Israeli BSF grant.}
}

\begin{document}
\maketitle

\begin{abstract}
An old problem of Moser asks: how large of a union-free subfamily
does every family of $m$ sets have? A family of sets is called \emph{union-free} if there are no three
distinct sets in the family such that the union of two of the sets is equal to the third set. We show that every
family of $m$ sets contains a union-free subfamily of size at least $\lfloor \sqrt{4m+1}\rfloor -1$ and that this bound is tight.
This solves Moser's problem  and proves a conjecture of Erd\H{o}s and
Shelah from 1972.

More generally, a family of sets is \emph{$a$-union-free} if there are no $a+1$ distinct sets in the
family such that one of them is equal to the union of $a$
others. We determine up to an absolute multiplicative constant factor the size of the largest
guaranteed $a$-union-free subfamily of a family of $m$ sets. Our result
verifies in a strong form a conjecture of Barat, F\"{u}redi, Kantor, Kim and Patkos.
\end{abstract}

\section{Introduction}

A set $A$ of integers is {\it sum-free} if there are no $x,y,z \in A$ such that $x+y=z$. Erd\H{o}s \cite{Er65}
in 1965 proved that every set of $n$ nonzero integers
contains a sum-free subset of size at least $n/3$. The proof is an influential application of the
probabilistic method in extremal number theory.
This result was rediscovered by Alon and Kleitman \cite{AK}, who showed how to find a sum-free subset of size at least  $(n+1)/3$.
Finally, Bourgain \cite{Bo} using harmonic analysis improved the lower bound to $(n+2)/3$. This result is the current
state of the art for this problem. It is not even known if
the constant factor $1/3$ is best possible.

The analogous problem in extremal set theory has also been studied for a long time. A family of sets is called
\emph{union-free} if there are no three distinct sets
$X,Y,Z$ in the family such that $X\cup Y=Z$. An old problem of Moser asks: how large of a union-free subfamily
does every family of $m$ sets have? Denote this number
by $f(m)$. The study of $f(m)$ has attracted considerable interest. Riddell observed that $f(m) \geq \sqrt{m}$
(this follows immediately from Dilworth's theorem, see
below). Erd\H{o}s and Koml\'{o}s \cite{ErKo} determined the
correct order of magnitude of $f(m)$ by proving that $f(m) \le 2\sqrt{2m}+4$. They conjectured that
$f(m)=(c-o(1))\sqrt{m}$ for some constant $c$, without specifying
the right value of $c$. In 1972, Erd\H{o}s and Shelah \cite{ErSh} improved both the upper and lower bounds
by showing that $\sqrt{2m}-1 < f(m) < 2\sqrt{m}+1$ (the lower bound was also obtained independently by Kleitman \cite{Kl73}).
Erd\H{o}s and Shelah conjectured that their upper bound is
asymptotically tight.
\begin{conjecture}
$f(m)=(2-o(1))\sqrt{m}$.
\end{conjecture}

We verify this conjecture and solve Moser's problem.

\begin{thm} \label{thm:main1_intro}
For all $m$, we have $$f(m)= \lfloor \sqrt{4m+1} \rfloor-1.$$
\end{thm}

Let $a \ge 2$ be an integer. A family of sets is called {\it $a$-union-free}
if there are no $a+1$ distinct sets $X_{1},\cdots,X_{a+1}$ such that $X_{1}\cup\cdots\cup X_{a}=X_{a+1}$.
Let $g(m,a)$ be the minimum over all families of $m$ sets of the size
of the largest $a$-union-free subfamily. In particular, $g(m,2)=f(m)$. The same proof which shows
$f(m) > \sqrt{2m}-1$ also shows that $g(m,a) > \sqrt{2m}-1$.
Recently, Barat, F\"{u}redi, Kantor, Kim and Patkos \cite{FuBaKaKiPa} proved that $g(m,a)\le c(a+a^{1/4}\sqrt{m})$ for some absolute constant $c$ and made the following conjecture on the growth of $g(m,a)$.

\begin{conjecture}
$\lim_{a\rightarrow\infty}\liminf_{m\rightarrow\infty}\frac{g(m,a)}{\sqrt{m}}=\infty$.
\end{conjecture}

We prove this conjecture in the following strong form, which further gives the correct order of magnitude for $g(m,a)$.

\begin{thm} \label{thm:main2_intro}
For all $m \geq a \geq 2$, we have $g(m,a) \ge \max\{a,a^{1/4}\sqrt{m}/3\}$.
\end{thm}

The lower bound in Theorem \ref{thm:main2_intro} is tight apart from an absolute multiplicative constant factor by the above mentioned upper bound from \cite{FuBaKaKiPa}. Of course, if $m \leq a$, we have trivially $g(m,a)=m$.

In the next two sections, we prove Theorem \ref{thm:main1_intro} and Theorem \ref{thm:main2_intro},
respectively. For the proofs of these theorems, it is helpful to study the structure of the partial order on sets
given by inclusion.
Recall that a chain (antichain) in a poset is a collection of pairwise comparable (incomparable) elements. Dilworth's
theorem \cite{Di}
implies that any poset with $m$ elements contains a chain or antichain of size at least $\sqrt{m}$. Notice that a
chain or antichain of sets is $a$-union-free for
all $a$, but the lower bound $g(m,a) \geq \sqrt{m}$ we get from this simple argument is not strong enough. For the
proof of Theorem \ref{thm:main2_intro}, we find
considerably larger structures in posets which imply that the subfamily is $a$-union-free. The existence of such
large structures in posets may be of independent interest.

\section{ Erd\H{o}s-Shelah conjecture}
In this section we prove that every family of $m$ sets contains a union-free subfamily of size at least
$\lfloor \sqrt{4m+1} \rfloor -1$. This verifies in a strong form the Erd\H{o}s-Shelah conjecture.
We finish the section with a modification of the construction of Erd\H{o}s and Shelah which shows that this bound is tight.

Let $\mathcal{F}$ be a family of $m$ sets and assume that the maximum
union-free subfamily has size $\alpha$. Let $t$ be the length of
the longest chain in $\mathcal{F}$. Let $\mathcal{F}_{t}$ be the
family of maximal (by inclusion) sets in $\mathcal{F}$, and for $i=t-1,\cdots,1$,
inductively define $\mathcal{F}_{i}$ as the family of maximal sets
in $\mathcal{F}\setminus(\bigcup_{j=i+1}^{t}\mathcal{F}_{j})$. We call
$\mathcal{F}_{i}$ the\emph{ $i$th level of $\mathcal{F}$}. Note that since there
exists a chain of length $t$, $\mathcal{F}_{i}$ is non-empty for
all $i$, and for every set in $\mathcal{F}_{i}$ there is a chain of length
$t-i+1$ starting at that set which hits every level above $\mathcal{F}_{i}$
exactly once. Furthermore, if $X \subset Y$,
then $X$ lies in a level below $Y$, and thus each level is an antichain.
For a set $X\in\mathcal{F}_{i}$,
let $i$ be the \emph{rank} of $X$.

We define an auxiliary graph $H=H(\mathcal{F})$ with vertex set $V(H)=\{1,\ldots,t\}$ as follows. The pair $(i,i')$ with $i<i'$
is an edge of $H$ if there exist two disjoint
chains $\mathcal{X}$ and $\mathcal{Y}$ which both start at $\mathcal{F}_{i}$ and
end at $\mathcal{F}_{i'-1}$ such that

\vspace{0.15cm}
{\bf (i)} $\mathcal{X},\mathcal{Y}$
both have length $i'-i$, and

\vspace{0.1cm}

{\bf (ii)} at most one set in $\{ X\cup Y \, | \,  X\in\mathcal{X},Y\in\mathcal{Y}\}$ has rank $i'$ and all the rest of the sets of this
form have rank greater than $i'$ or do not belong to family $\cal F$.
\vspace{0.15cm}

\noindent
Thus, by definition, we can see that if $(i,i')$ is an edge of $H$ and $i<i''<i'$, then
$(i,i'')$ and $(i'', i')$ are also edges of $H$, i.e., we have \emph{monotonicity} of $H$.

We first prove that if there are many pairs of sets $X,Y \in \mathcal{F}$ such that
$X\cup Y\in\mathcal{F}$, then there are many edges of $H$.
\begin{lem}
\label{prop:union_or_2connected}Let $i<i'$ and $X,Y$ be two sets
of rank at most $i$ such that $X\cup Y\in\mathcal{F}_{i'}$.
Then $(i,i')$ is an edge of $H$.\end{lem}
\begin{proof}
Let $X,Y$ be sets which have rank $a,b$ respectively ($a,b\le i$),
and whose union lies in $\mathcal{F}_{i'}$. Let $\mathcal{X}=\left\{ X_{a}=X,X_{a+1},\cdots,X_{i'-1}\right\} $
be a chain starting at $X$ and ending in $\mathcal{F}_{i'-1}$.
Similarly define a chain $\mathcal{Y}=\left\{ Y_{b}=Y,Y_{b+1},\cdots,Y_{i'-1}\right\}$ starting at $Y$ and ending in
$\mathcal{F}_{i'-1}$. If there exists an element $Z$ which lies in both $\mathcal{X}$
and $\mathcal{Y}$, then $X_{a},Y_{b}\subset Z$ and thus the rank
of $X \cup Y=X_{a}\cup Y_{b}$ is at most the rank of $Z$ which is at most
$i'-1$. This contradicts our assumption that $X\cup Y\in\mathcal{F}_{i'}$.
Therefore $\mathcal{X}$ and $\mathcal{Y}$ are disjoint. Moreover,
for every $X_{c}\in\mathcal{X}$ and $Y_{d}\in\mathcal{Y}$, we have
$X\cup Y\subset X_{c}\cup Y_{d}$ and since $X\cup Y$ has rank $i'$,
we know that $X_{c}\cup Y_{d}$ either is equal to $X\cup Y$, has
rank larger than $i'$, or is not in $\mathcal{F}$. By taking a subchain $\{X_i, \ldots, X_{i'-1}\}$ of $\mathcal{X}$
and a subchain $\{Y_i, \ldots, Y_{i'-1}\}$  of $\mathcal{Y}$, we can see that $(i,i')$ is an edge of $H$.
\end{proof}

We will show that if $\alpha$ is small, then the total number
of sets cannot be $m$, and this will give us a contradiction.
Thus we need some tools which allow us to bound the size of the levels. The next lemma shows that by
collecting the levels into groups whose indices form independent sets in $H$,
we can obtain a good bound on the size of the levels.
\begin{lem}
\label{lemma:groupsize}Let $a_1<\ldots<a_k$ be the vertices of an independent set in $H$. Then $\sum_{j=1}^{k}|\mathcal{F}_{a_{j}}|\le\alpha.$
Further, if $(a_1,b)$ is an edge of $H$ with $b<a_1$, then
\[\sum_{j=1}^{k}|\mathcal{F}_{a_{j}}|\le\alpha-2(a_{1}-b)+1.\]
\end{lem}
\begin{proof}
The first assertion is a straightforward corollary of Lemma
\ref{prop:union_or_2connected}, as the lemma implies that
$\bigcup_{j=1}^{k}\mathcal{F}_{a_{j}}$ is a union-free subfamily.

Let $\mathcal{X}$ and $\mathcal{Y}$ be two chains guaranteed by the fact $(b,a_1)$ is an edge of $H$.
Let $Z$ be the unique set (if it exists) in $\left\{ X \cup Y| X\in\mathcal{X}, Y\in\mathcal{Y}\right\} $
which has rank $a_{1}$. We claim that $\mathcal{W}=\mathcal{X}\cup\mathcal{Y} \cup\Big(\bigcup_{j=1}^{k}\mathcal{F}_{a_{j}}\Big) \setminus \left\{ Z\right\}$ is a union-free subfamily. Assume that this is not the case and let
$W_{1},W_{2},W_{3}\in\mathcal{W}$ be three distinct sets in $\mathcal{W}$ such
that $W_{1}\cup W_{2}=W_{3}$. If $W_{3}\in\mathcal{F}_{a_{i}}$ for
some $i \ge 2$, then since each level is an antichain we know that
both $W_{1}$ and $W_{2}$ has rank at most $a_{i-1}$. However, by
Lemma \ref{prop:union_or_2connected} this would contradict the
assumption that $a_{i-1}$ and $a_i$ are not adjacent in $H$. We may therefore assume that $W_{3}\in\mathcal{X}\cup\mathcal{Y}\cup\mathcal{F}_{a_{1}}$.
If $W_{3}\in\mathcal{F}_{a_{1}}$, then $W_{1},W_{2}$ must have smaller
rank than $a_{1}$ and thus lie in $\mathcal{X}\cup\mathcal{Y}$. As $W_1,W_2,W_3$ are distinct, $W_3=W_1 \cup W_2$, and $\mathcal{X}$ and $\mathcal{Y}$ are chains, the sets $W_1$, $W_2$ cannot both lie in $\mathcal{X}$ or both lie in $\mathcal{Y}$. Recall that (by condition (ii) above) $Z$ was the only
set in $\{ X\cup Y \, | \,  X\in\mathcal{X},Y\in\mathcal{Y}\}$ with rank $a_{1}$. Since $W_{3}\neq Z$ and $W_3=W_1 \cup W_2$, the rank of $W_3$ is not $a_1$. If $W_{3}\in\mathcal{X}$, then since $W_{1},W_{2},W_{3}$
are distinct and $\mathcal{X}$ is a chain, one of the sets $W_{1}$ or $W_{2}$ must lie in $\mathcal{Y}$.
Without loss of generality assume that $W_{1}\in\mathcal{Y}$. This
implies that $W_{1}\cup W_{3}=W_{3}$ which is impossible again by condition
(ii) since all the sets in $\{ X\cup Y \, | \,  X\in\mathcal{X},Y\in\mathcal{Y}\}$ have rank at least $a_1$ and the rank of
$W_3$ is smaller than $a_1$. Similarly,
$W_{3}\notin\mathcal{Y}$. Therefore,
there cannot exist such sets $W_{1},W_{2},W_{3}$, and $\mathcal{W}$
is a union-free subfamily whose size satisfies \[
\sum_{j=1}^{k}|\mathcal{F}_{a_{j}}|+2(a_{1}-b)-1 \leq |\mathcal{W}| \le\alpha,\]
which implies our claimed inequality.
\end{proof}

Now we are prepared to prove Theorem \ref{thm:main1_intro}. The main idea
is to properly color the vertices of $H$ using as few colors as possible
(i.e., partition the vertex set of $H$ into as few independent sets as possible) in order to maximize the power
of Lemma \ref{lemma:groupsize}. Consider the following greedy algorithm
for finding such a partition into independent sets. The first independent set $I_1$ contains $1$.
We find the least $a_2^1$ which is not adjacent to $1$ and add it to $I_1$. After finding the $j$th element $a_j^1$ of $I_1$, we find the least
$a_{j+1}^1>a_j^1$ which is not adjacent to $a_j^1$ and add it to $I_1$. We continue this procedure until we cannot add any more vertices.
Note that by the monotonicity condition satisfied by graph $H$, the set $I_1$ is indeed an independent set.
Assume that we finished constructing $I_i$. The independent set $I_{i+1}$ contains the least $a_1^{i+1}$ not in any of the previous independent sets.
After finding the $j$th element $a_j^{i+1}$ of $I_{i+1}$, we find the least
$a_{j+1}^{i+1}>a_j^{i+1}$ which is not in any of the previously chosen independent sets and is not adjacent to $a_j^{i+1}$ and add it to $I_{i+1}$. We continue this procedure until we cannot add any more vertices. Note that by the monotonicity condition satisfied by graph $H$, the set $I_{i+1}$ is indeed an independent set. We continue picking out independent sets until there are no remaining vertices of $H$.

By the first part of Lemma \ref{lemma:groupsize}, the
sum of the sizes of the levels with index in $I_1$ satisfies $\sum_{j \in I_1}|\mathcal{F}_{j}|\le\alpha$.
For the other independent sets we can use the second part of Lemma \ref{lemma:groupsize}
to obtain a better bound.
\begin{lem}\label{great}
For all $i>1$, $\sum_{j\in I_i}|\mathcal{F}_{j}|\le\alpha-2i+3$.\end{lem}
\begin{proof}
Let $a$ be the least element in $I_i$. For each $j<i$, let $a_j$
be the largest vertex in $I_j$ with $a_j<a$. Since $a \not \in I_j$ by the greedy algorithm, we know that $a_j$ is adjacent to $a$ in graph $H$.
Since there are $i-1$ independent sets $I_j$ with $j<i$, there is a vertex $b$ adjacent to $a$ in graph $H$ such that $a-b\ge i-1$. Thus, by the second part of Lemma \ref{lemma:groupsize}, we have \[
\sum_{j \in I_i}|\mathcal{F}_{j}|\le\alpha-2(i-1)+1=\alpha-2i+3.\]
\end{proof}

\noindent
Lemma \ref{great} in particular implies that the total number of independent sets is at
most $\lfloor\frac{\alpha+3}{2}\rfloor$.

\begin{proof}[Proof of Theorem \ref{thm:main1_intro}]
We first prove the lower bound on $f(m)$. The proof splits into two cases.
Suppose first that $n^2 \le m < n^2+n$ for some integer $n$ so that $\lfloor \sqrt{4m+1} \rfloor=2n$. Assume for contradiction that the size $\alpha$ of the
largest union-free subfamily satisfies $\alpha \le 2n-2$.

This implies with Lemma \ref{great} that \[ m=\sum_{i} \Big|\bigcup_{j \in I_i} \mathcal{F}_{j}\Big|\le(2n-2)+\sum_{i=2}^{n}(2n-2i+1)=(2n-2)+(n-1)^{2}=n^{2}-1\]
which contradicts $m \geq n^2$. Hence, in this case we must have $\alpha \ge 2n-1=\lfloor \sqrt{4m+1} \rfloor-1$.

We may therefore assume that $n^{2}+n\le m<(n+1)^{2}$ for some integer $n$ so that $\lfloor \sqrt{4m+1} \rfloor=2n+1$. Assume for contradiction that
$\alpha \leq 2n-1$. This implies with Lemma \ref{great} that \[
m=\sum_{i} \Big|\bigcup_{j \in I_i} \mathcal{F}_{j}\Big| \le(2n-1)+\sum_{i=2}^{n+1}(2n-2i+2)=(2n-1)+n^{2}-n=n^{2}+n-1\]
which contradicts $m \geq n^2+n$ and we must have $\alpha\ge 2n = \lfloor \sqrt{4m+1} \rfloor - 1$. This completes the proof of the lower bound in Theorem \ref{thm:main1_intro}.

We use the following construction given by Erd\H{o}s and Shelah \cite{ErSh} to show that the lower bound on $f(m)$ is tight. We begin with the
case $m=n^2$ is a perfect square. Consider the collection  $\cal F$ of $n^2$ sets $X_{ij}=\{x \in \mathbb{N} \, | \, n+1-i \leq x \leq n+j \}$ where
$1\leq i,j \leq n$. We may assume that each set $X_{ij}$
is placed on the $(i,j)$ position of the $n \times n$ grid.
Let $\mathcal{F}'$ be a union-free subfamily of $\cal F$. Note that $\mathcal{F}'$ cannot contain a
triple $X_{i'j}, X_{ij'}, X_{ij}$ with $i'<i$ and $j'<j$ since $X_{i'j} \cup X_{ij'}=X_{ij}$.
Delete from each column in the grid the bottommost set in $\mathcal{F}'$.
Note that we removed at most $n$ sets and after removing these sets, there will be no set of $\mathcal{F}'$ remaining in the lowest row of the grid. Then remove from each row in the grid the leftmost remaining set
in $\mathcal{F}'$. This removes at most $n-1$ additional sets. Now there cannot be any remaining set as
otherwise $\mathcal{F}'$ will contain some triple $X_{i'j}, X_{ij'}, X_{ij}$ with $i'<i$ and $j'<j$.
Since we removed at most $n+n-1=2n-1$ sets, this
implies that $\mathcal{F}'$ has size at most $2n-1$. Thus our theorem is tight
for $m=n^2$. We can modify this construction by taking $\cal F$ to be the collection $X_{ij}=\{x \in \mathbb{N} \, | \, n+1-i \leq x \leq n+j \}$ with
$1\leq i\leq n+1$ and $1\leq j \leq n$. In this way we will have a
total of $m=n^2+n$ sets, the corresponding grid will have size $(n+1)\times n$, and the maximum union-free subfamily will have size
at most $2n$. This again shows that our theorem is tight for $m=n^2+n$. Note that, in this construction, $X_{11}$ is less than all other sets, and hence every maximal
union-free subfamily of $\mathcal{F}$ contains $X_{11}$. Therefore,  deleting $X_{11}$ from $\mathcal{F}$ decreases the number of sets in it by one, but also
decreases by one the size of the maximum union-free subfamily. This shows that $f(n^2-1) \leq 2n-2$ and $f(n^2+n-1) \leq 2n-1$.
Of course, $f(m)$ is a monotone increasing function, and this gives the upper bound on $f(m)$ for all other values of $m$.
\end{proof}

\section{$a$-union-free subfamilies}

In this section we prove Theorem \ref{thm:main2_intro}, which says that every family of $m$ sets contains an
$a$-union-free subfamily of size at least $\max\{a,a^{1/4}\sqrt{m}/3\}$. We will
in fact prove a stronger theorem for general posets which implies
Theorem \ref{thm:main2_intro}. In a poset $P$, a
{\it $(\ell,\alpha)$-ladder} consists of a chain $X_1<\ldots < X_{\ell}$,
and $\ell$ antichains $\mathcal{Y}_1,\ldots,\mathcal{Y}_{\ell}$ each
of order $\alpha$ such that each $Y \in \mathcal{Y}_i$ satisfies $Y
\leq X_i$, and for each $i$ and each $Y \in \mathcal{Y}_{i+1}$, $Y
\not \leq X_i$. For $\alpha \geq 1$, the {\it size} of an $(\ell,\alpha)$-ladder is the
number of elements in $\mathcal{Y}_1 \cup \ldots \cup
\mathcal{Y}_{\ell}$, which is $\ell\alpha$. The size of an $(\ell,0)$-ladder we define to be $\ell$.
Notice that a chain of length $\ell$ forms a $(l,0)$-ladder.

\begin{lem} \label{lemma:ladder}
An $(l,\alpha)$-ladder in a family of sets contains an $a$-union-free subfamily of size at least the size of the ladder, which is $l \cdot \max\{\alpha,1\}$.
\end{lem}
\begin{proof}
If $\alpha \le 1$, then we can just take the chain as the $a$-union-free
subfamily.
Otherwise, we claim that the collection $\mathcal{Y}=\bigcup_{i=1}^{l}\mathcal{Y}_{i}$
is an $a$-union-free family and this proves the lemma. Consider a
linear extension of the partial order on the sets given by inclusion. Let
$Y_{1},\cdots,Y_{a}$ be $a$ distinct sets in $\mathcal{Y}$ and assume without loss of generality that
$Y_{a}\in\mathcal{Y}_{t}$ is the largest of these $a$ sets
in this linear extension. By the definition of a ladder, each of these sets is a subset of $X_t$ and therefore the
union
$Y_{a+1}:=Y_1 \cup \ldots \cup Y_a$ of these $a$ sets is a subset of $X_t$. For $j>t$, we have $Y_{a+1} \not \in
\mathcal{Y}_j$ as no set in $\mathcal{Y}_j$ is a subset of $X_t$. Since $Y_a \subset Y_{a+1}$ and $\mathcal{Y}_t$ is an
antichain, $Y_{a+1} \not \in \mathcal{Y}_t$. If $Y_{a+1} \in \mathcal{Y}_j$ with $j<t$, then $Y_a \subset Y_{a+1}
\subset X_{t-1}$, which contradicts $Y_a \in \mathcal{Y}_t$ and no set in $\mathcal{Y}_t$ is a subset of $X_{t-1}$.
Thus $Y_{a+1} \not \in \mathcal{Y}$ and hence $\mathcal{Y}$ is $a$-union-free.
\end{proof}

In a poset $P$, a set $S$ of elements is {\it $a$-degenerate} if no element of $S$ is larger than $a$ other elements of $S$. It is easy to check that if $P$ is a family of sets with partial order defined by inclusion, then any $a$-degenerate subset is $a$-union-free. Thus in order to prove Theorem \ref{thm:main2_intro}, it is enough to establish the following theorem.

\begin{thm}\label{key}
Every poset $P$ on $m \geq a$ elements contains a subset of size at least $\max\{a,a^{1/4}\sqrt{m}/3\}$ which is $a$-degenerate or forms a
ladder.
\end{thm}

The bounds in Theorems \ref{thm:main2_intro} and \ref{key} are tight up to an absolute constant factor for all $m$ and
$a$ by a construction of Barat et al. \cite{FuBaKaKiPa} which we will give at the end of this section.

\begin{proof}
Let $c=1/3$. If $m \leq c^{-2}a^{3/2}$, then we can take any $a$ elements as an $a$-degenerate set.
Since $a \geq ca^{1/4}\sqrt{m}$ we have nothing to prove in this range. Thus we will assume that
$m > c^{-2}a^{3/2}$. We will show that there exists a subset of size at least $ca^{1/4}\sqrt{m}$ which is $a$-degenerate or forms a ladder.
Since every poset on $m$ elements contains a chain or an antichain of size at least $\sqrt{m}$, and a chain forms a ladder and an antichain is $a$-degenerate,
we can also assume that $a>81$.

We may assume that the height of poset $P$ is less than $ca^{1/4}\sqrt{m}$ as otherwise we can take the
longest chain as the ladder. In each step, we will either find the desired subset, or we will delete some elements and
the height of the remaining poset will drop significantly compared to the drop in the number of remaining elements.

Let $P_0=P$, $h_0 = ca^{1/4}\sqrt{m}$, and $m_0 = m$. Suppose that after $i$ steps the remaining poset $P_i$ has at
least $m_i$ elements and height at most $h_i$ where $m_i = 2^{-i}m_0$ and $h_i = 2^{-2i}h_0$. Notice that this
condition is satisfied for $i=0$. For technical reason we need additional parameters $x_i$. Let $x_0=2^{T}$ for some
integer $T$ so that $2^T = \beta c^2 a^{1/2}$ with $\frac{9}{4} \leq \beta < \frac{9}{2}$ and let $x_i = 2^{-i}x_0$. Note that
such $T$ indeed exists since $\frac{9}{2}c^2 a^{1/2}>1$, and $x_i$ is an integer for
all $i=0,\cdots, T$.

Suppose we have just finished step $i \leq T-1$. Partition $P_i$ into levels $\mathcal{F}_1 \cup \ldots \cup
\mathcal{F}_{\lfloor h_i \rfloor}$,
where $\mathcal{F}_j$ is the minimal elements of the subposet of $P_i$ formed by deleting all $\mathcal{F}_k$ with
$k<j$ (some of the levels in the end can be empty). Partition the poset $P_i$ into sets $S_k$ each consisting of $x_i$
consecutive levels of $P_i$ (except possibly one set which consists of less than $x_i$ consecutive levels). Call an element $p \in S_k$ {\it nice} if it is larger than $a$ other elements of $S_k$,
and {\it bad} otherwise. For each $k$, the bad elements within $S_k$ form an $a$-degenerate set, and hence we may
assume $S_k$ contains less than $ca^{1/4}\sqrt{m}$ bad elements. Thus, the total number of bad elements is at most the
number of sets $S_k$ times $ca^{1/4}\sqrt{m}$, which is
\[ \left\lceil \frac{h_i}{x_i} \right\rceil ca^{1/4}\sqrt{m} =
\left\lceil \frac{2^{-2i}h_0}{2^{-i}x_0} \right\rceil ca^{1/4}\sqrt{m} \le \left( \frac{ca^{1/4}\sqrt{m}}{2^{i} x_0} +
1 \right) ca^{1/4}\sqrt{m} = \frac{m}{\beta \cdot2^{i}} + ca^{1/4}\sqrt{m}. \]
By $m \ge c^{-2}a^{3/2}$ and $i \leq T-1$, we have
$$ca^{1/4}\sqrt{m} =  ca^{1/4}\frac{m}{\sqrt{m}} \leq c^2a^{-1/2}m=\beta c^4 \frac{m}{2^T} \leq \frac{\beta c^4}{2} \frac{m}{2^i}.$$
Since $c=1/3$ and $\frac{9}{4} \leq \beta \leq \frac{9}{2}$, one can easily check $\frac{1}{\beta}+\frac{\beta c^4}{2} \leq
4/9+1/36<1/2$, so
the number of bad elements in $P_i$ is at most $m_i/2$.

Let $P_{i+1}$ be the subposet of $P_i$ consisting of the nice elements in $P_i$.
Since there are at most $m_i/2$ bad elements in $P_i$, the total number of
elements in $P_{i+1}$ is at least $m_i - m_i/2 = m_{i+1}$. If the
height of this poset is at least $h_{i+1}$, then by definition, we
can find a chain of length at least $h_{i+1}$. Let $X_1 < \ldots <
X_{\ell}$ be a subchain of this chain such that each $X_j$ belongs
to different sets $S_k$. Since $S_k$ is a union of $x_i$ levels, we
can guarantee such a subchain of length at least $\ell \ge
\left\lceil h_{i+1}/x_i \right\rceil$. By the definition of nice
elements, each element $X_j \in S_k$ is greater than $a$ other elements
in $S_k$. Therefore, $X_j$ is greater than $\alpha=\lceil a/x_i
\rceil$ elements which all belong to a single level of $S_k$. By the
construction of levels, these elements form an antichain. Denote
this antichain by $\mathcal{Y}_j$. We have thus constructed an
$(\ell,\alpha)$-ladder in $P$. The size of this ladder is
\begin{eqnarray*}
\ell\alpha & \geq & \left\lceil \frac{h_{i+1}}{x_i} \right\rceil
\cdot \left\lceil \frac{a}{x_i} \right\rceil \geq \frac{2^{-2i-2}h_0
\cdot a}{2^{-2i}x_0^2} = \frac{ca^{1/4}\sqrt{m}}{4 \cdot \beta^2c^4} \geq
ca^{1/4}\sqrt{m} ,\end{eqnarray*}
where we used that $x_0 = 2^T = \beta c^2a^{1/2}$, $\beta < 9/2$, $h_0=ca^{1/4}\sqrt{m}$ and $c=1/3$.
Thus we find a sufficiently large ladder to complete this case.

Otherwise the height of the poset $P_{i+1}$ is less than $h_{i+1}$, and since it contains at least $m_{i+1}$ elements and has height at most $h_{i+1}$, we can move on to the next step. Assume that the
process continues until $i=T$. Then we have a poset $P_T$ which has at least $m_T$ elements and has
height at most $h_T$. In this case, there exists an antichain of size at least
\[ \left\lceil \frac{m_T}{h_T} \right\rceil \ge \frac{2^{-T}m_0}{2^{-2T}h_0} = \frac{2^T m}{ca^{1/4}\sqrt{m}} \ge
\frac{c^2a^{1/2} m}{ca^{1/4}\sqrt{m}} =
ca^{1/4}\sqrt{m}.\]
An antichain is an $a$-degenerate set, so this concludes the proof.
\end{proof}

We end this section with the construction given in \cite{FuBaKaKiPa} of a family of sets which does not
contain a large $a$-union-free subfamily. Let $k = \lceil \sqrt{a} \rceil$.
Let $\mathcal{F}_1$ be the family of sets
given in the end of the previous section, i.e., it contains sets
$X_{ij}=\{x \in \mathbb{N} \, | \, n+1-i \leq x \leq n+j\}$ for all $1\leq i,j \leq n$.
A similar argument as in there shows that the largest $a$-union-free subfamily of $\mathcal{F}_1$ has
size at most $2k n$. Indeed we again assume that each set $X_{ij}$
is placed on the $(i,j)$ position of the $n \times n$ grid.
Let $\mathcal{F}'_1$ be an $a$-union-free subfamily of ${\cal F}_1$.
Delete from each column in the grid the bottommost $k$ sets in $\mathcal{F}'_1$.
Note that we removed at most $kn$ sets.
Then remove from each row in the grid the $k$ leftmost remaining sets in $\mathcal{F}'_1$. This removes at most $kn$ additional sets. Now there cannot be any
remaining set. Otherwise $\mathcal{F}'_1$ will contain sets $X_{i_1j}, \ldots, X_{i_kj}$ and  $X_{ij}$ with $i_{\ell}<i$ for all $1 \leq \ell \leq k$, such that
$ X_{ij}$ and each set $X_{i_{\ell}j}$ have at least $k$ other sets from $\mathcal{F}'_1$ in their column below them. This gives at least $k^2 \geq a$ sets (all properly contained in $X_{ij}$) from which one can easily choose $a$ sets whose union equals to $X_{ij}$. Since we removed at most $kn+kn=2kn$ sets, this
implies that $\mathcal{F}'_1$ has size at most $2kn$.

Note that set $X_{nn}$ contains all the other sets in
$\mathcal{F}_1$. Let $\mathcal{G}_2$ be a family of sets constructed in the
same manner but over a different (disjoint) universe of elements from $\mathcal{F}_1$,
and let $\mathcal{F}_2$ be the family $\{G \cup X_{nn} | G \in \mathcal{G}_2 \}$.
Thus every set in $\mathcal{F}_2$ contains all the sets in $\mathcal{F}_1$ and the $a$-union-free subfamily of $\mathcal{F}_2$ also has
size at most $2k n$.
Repeat this process and construct families
$\mathcal{F}_1, \mathcal{F}_2, \cdots, \mathcal{F}_k$ such that for all $\ell$,
every set in $\mathcal{F}_{\ell+1}$
contains all the sets in $\mathcal{F}_{\ell}$. Let $\mathcal{F}=\bigcup_{\ell=1}^{k} \mathcal{F}_{\ell}$ be our
family.

Then the number of sets in $\mathcal{F}$ is $kn^2$. We will use this family to obtain a bound on $g(m,a)$ for all values of $m$ such that
$k(n-1)^2 < m \le k n^2$. We can bound the size of the largest $a$-union-free subfamily $\mathcal{F}'$ of $\cal F$ as follows. Assume that $\ell$ is the first index such that there are more than $a$ sets of $\mathcal{F}'$ in $\mathcal{F}_1 \cup \cdots \cup \mathcal{F}_{\ell}$. Then in $\mathcal{F}'$ there
are less than $a$ sets in the levels up to $\mathcal{F}_{\ell-1}$, at most $2kn$ sets in $\mathcal{F}_{\ell}$, and it is easy to see for each $t > \ell$ that the sets in $\mathcal{F}'$ from $\mathcal{F}_t$ form a 2-union-free subfamily. Thus there are at most $2n$ such sets. Therefore the number of sets in $\mathcal{F}'$ is at most
\[ a + 2kn + (k-i)\cdot 2n \le a + 4k + 4k(n-1) \le 5a + 4(a^{1/4}+1)\sqrt{m}, \]
where we used the inequalities $m > k(n-1)^2$ and $a + 4\lceil \sqrt{a} \rceil \le 5a$ for all integers $a \ge 1$. This shows that the bound in
Theorem \ref{thm:main2_intro} is tight up to the constant factor for all $a$ and $m$.

\section{Concluding Remarks}

In this paper, we solved Moser's problem by determining the largest $f(m)$ for which every family of $m$ sets has a union-free subfamily of size $f(m)$. We do not know whether or not the family of sets given by Erd\H{o}s and Shelah \cite{ErSh} is essentially the only extremal family which shows that this bound is tight. It would be interesting to further study this problem of classifying the extremal families. Specifically, can one classify the extremal families which achieve the bound $f(m)=\lfloor \sqrt{4m+1}\rfloor -1$?

We also determined up to an absolute constant factor the largest $g(m,a)$ for which every family of $m$ sets has an $a$-union-free subfamily of size $g(m,a)$.
We did not try to optimize constants in the proof of Theorem \ref{thm:main2_intro} for the sake of clarity of presentation. Although the bound in this theorem can be further improved, some new ideas are needed to determine the asymptotic behavior of $g(m,a)$ for $a \geq 3$.

The proofs of Theorems \ref{thm:main1_intro} and \ref{thm:main2_intro} giving lower bounds for $f(m)$ and $g(m,a)$ can easily be made algorithmic. That is, we can
find a union-free or $a$-union-free subfamily of a family of $m$ with size guaranteed by these theorems in polynomial time.

There are other directions that have been studied concerning union-free subfamilies.
For example, Abbott and Hanson \cite{AH} proposed studying the minimum number of colors necessary to color the subsets of an $n$-element set so that each color
class is union-free. This problem was further studied by
Aigner, Duffus, and Kleitman \cite{AiDuKl}. A related result of Kleitman \cite{Kleitman}, improving on an earlier paper \cite{Kl66},
solved a conjecture of Erd\H{o}s by showing that the largest union-free family of subsets of an $n$-element set has cardinality at most $\left(1+\frac{c}{\sqrt{n}}\right){n \choose \lfloor
n/2 \rfloor}$, where $c$ is an absolute constant. We can also replace the union-free condition by other conditions.
For example, Gunderson, R\"odl, and Sidorenko \cite{GuRoSi} studied
the maximum cardinality of a family of subsets of an $n$-element set which does not contain a $d$-dimensional Boolean algebra. A related problem of
Erd\H{o}s and Shelah was solved by Barat, F\"{u}redi, Kantor, Kim, and Patkos \cite{FuBaKaKiPa},
who estimated the size of a maximum subfamily of a family of $m$ sets which does not contain a $d$-dimensional Boolean algebra.


\begin{thebibliography}{99}
\bibitem{AH}
H. L. Abbott and D. Hanson, A problem of Schur and its generalizations, {\it Acta Arith.} {\bf 20} (1972), 172--185.

\bibitem{AiDuKl}
M. Aigner, D. Duffus, and D. Kleitman,
\newblock Partitioning a power set into union-free classes,
\newblock {\em Discrete Math.} 88 (1991), 113--119.

\bibitem{AK}
N. Alon and D. Kleitman,
Sum-free subsets, in: {\em  A tribute to Paul Erd\H{o}s}, Cambridge Univ. Press, Cambridge, 1990, 13--26.

\bibitem{FuBaKaKiPa}
J. Barat, Z. F\"{u}redi, I. Kantor, Y. Kim, and B.Patkos,
\newblock Large $B_{d}$-free subfamilies,
\newblock arXiv:1012.3918v1 [math.CO].

\bibitem{Bo}
J. Bourgain, Estimates related to sumfree subsets of sets of integers, {\it Israel J. Math.} {\bf 97} (1997), 71--92.

\bibitem{Di}
R. P. Dilworth, A decomposition theorem for partially ordered sets, {\it Annals of Math.} {\bf 51} (1950), 161--166.

\bibitem{Er61}
P. Erd\H{o}s, Some unsolved problems, {\it Magyar Tud. Akad. Mat. Kutat\'o Int. K\"ozl.} {\bf 6} (1961), 221--254.

\bibitem{Er65}
P. Erd\H{o}s, Extremal problems in number theory, {\it Proceedings of Symposia in Pure Mathematics} {\bf 3} (1965), 181--189.

\bibitem{ErKo}
P. Erd\H{o}s and J. Koml\'{o}s,
\newblock On a problem of Moser,
\newblock Combinatorial theory and its applications I, (Edited by P. Erd\H{o}s, A. Renyi, and
V. Sos), North Holland Publishing Company, Amsterdam (1969).

\bibitem{ErSh}
P. Erd\H{o}s and S. Shelah,
\newblock On a problem of Moser and Hanson,
\newblock Graph theory and applications, Lecture Notes in Math., Vol. 303, Springer,
Berlin (1972), 75--79.

\bibitem{GuRoSi}
D. Gunderson, V. R\"{o}dl, and A. Sidorenko,
\newblock Extremal problems for sets forming boolean algebras and complete partite hypergraphs,
\newblock {\em J. Combin. Theory Ser. A} 88 (1999), 342--367.

\bibitem{Kl66}
D. Kleitman, On a combinatorial problem of Erd\H{o}s, {\it Proc. Amer. Math. Soc.} {\bf 17} (1966), 139--141.

\bibitem{Kleitman}
D. Kleitman,
\newblock Collections of subsets containing no two sets and their union,
\newblock Combinatorics, Amer. Math. Soc., Providence, R.I. (1971), 153--155.

\bibitem{Kl73} D. Kleitman, review of the article \cite{ErKo},
Mathematical Reviews MR0297582 (45 \#6636), 1973.


\end{thebibliography}
\end{document}